\documentclass[11pt,a4paper]{amsart}
\usepackage[foot]{amsaddr}
\usepackage[utf8]{inputenc}
\usepackage[english]{babel}
\usepackage[dvipsnames]{xcolor}
\usepackage{tikz}
\usepackage{soul,color}

\newcommand*{\fullref}[1]{\hyperref[{#1}]{\ref*{#1}. \nameref*{#1}}} 

\usepackage[osf]{mathpazo}
\usepackage{charter}
\usepackage[a4paper, top=3.2cm, bottom=3.2cm,left=3.2cm, right=3.2cm, heightrounded,bindingoffset=0mm]{geometry}

\usepackage[colorlinks=true]{hyperref}
\hypersetup{bookmarksnumbered=true, linkcolor=due, citecolor=Green, urlcolor=black,}
\makeatletter
\renewcommand{\email}[2][]{%
  \ifx\emails\@empty\relax\else{\g@addto@macro\emails{,\space}}\fi%
  \@ifnotempty{#1}{\g@addto@macro\emails{\textrm{(#1)}\space}}%
  \g@addto@macro\emails{#2}%
}
\makeatother
\usepackage{fancyhdr}

\usepackage{amsfonts}
\usepackage{amssymb}
\usepackage{amsthm}
\usepackage{amscd}
\usepackage{amsmath}
\usepackage{mathtools}
\usepackage{mathrsfs}
\usepackage{float}
\usepackage[all,cmtip]{xy}

\usepackage{color}	%scrivere a colori
\definecolor{due}{RGB}{0,76,147}

\usepackage{graphicx}	
\usepackage{multicol}	
\usepackage{wrapfig}
\usepackage[english]{babel} 
\usepackage[utf8]{inputenc}
\usepackage[T1]{fontenc}
\usepackage{enumitem}

\usepackage{longtable}
\usepackage{cite}

\theoremstyle{definition}
\newtheorem{defi}{Definition}[section]
\theoremstyle{plain}
\newtheorem{thm}[defi]{Theorem}

\newtheorem{prop}[defi]{Proposition}
\newtheorem{cor}[defi]{Corollary}
\newtheorem{lemma}[defi]{Lemma}
\theoremstyle{remark}

\newtheorem{rmk}[defi]{Remark}

\newtheorem*{notation}{Notation}

\theoremstyle{definition}

\newtheorem*{ack}{Acknowledgement}

\newcommand{\longra}{\longrightarrow}

\newcommand{\kd}{{\mathcal D}}
\newcommand{\ke}{{\mathcal E}}

\newcommand{\kh}{{\mathcal H}}

\newcommand{\kk}{{\mathcal K}}
\newcommand{\kl}{{\mathcal L}}
\newcommand{\km}{{\mathcal M}}

\newcommand{\ko}{{\mathcal O}}

\newcommand{\ks}{{\mathcal S}}

\newcommand{\kx}{{\mathcal X}}
\newcommand{\ky}{{\mathcal Y}}

\newcommand{\IC}{{\mathbb C}}

\newcommand{\IF}{{\mathbb F}}

\newcommand{\IP}{{\mathbb P}}

\newcommand{\Z}{{\mathbb Z}}

\newcommand{\gp}{{\mathfrak p}}

\newcommand{\gX}{{\mathfrak X}}

\newcommand{\rH}{{\rm H}}
\newcommand{\rHet}[1]{{\rm H}_{\rm \acute{e}t}^{#1}}

\newcommand{\Aut}{\operatorname{Aut}}

\newcommand{\Pic}{\operatorname{Pic}}
\newcommand{\NS}{\operatorname{NS}}

\newcommand{\Ext}{\operatorname{Ext}}

\newcommand{\Hom}{\operatorname{Hom}}

\newcommand{\Spec}{\operatorname{Spec}}

\newcommand{\id}{\operatorname{id}}

\newcommand{\ol}{\overline}

\newcommand{\wh}{\widehat}
\newcommand{\Isom}{\operatorname{Isom}}

  \makeatletter
\newcommand{\xdashrightarrow}[2][]{\ext@arrow 0359\rightarrowfill@@{#1}{#2}}
\newcommand{\xdashleftarrow}[2][]{\ext@arrow 3095\leftarrowfill@@{#1}{#2}}
\newcommand{\xdashleftrightarrow}[2][]{\ext@arrow 3359\leftrightarrowfill@@{#1}{#2}}
\def\rightarrowfill@@{\arrowfill@@\relax\relbar\rightarrow}
\def\leftarrowfill@@{\arrowfill@@\leftarrow\relbar\relax}
\def\leftrightarrowfill@@{\arrowfill@@\leftarrow\relbar\rightarrow}
\def\arrowfill@@#1#2#3#4{%
  $\m@th\thickmuskip0mu\medmuskip\thickmuskip\thinmuskip\thickmuskip
   \relax#4#1
   \xleaders\hbox{$#4#2$}\hfill
   #3$%
}
\makeatother

\begin{document}
	\title[On ordinary Enriques surfaces in positive characteristic]{On ordinary Enriques surfaces in positive characteristic}
	\author{Roberto Laface}
	\address[A1]{Technische Universit\"at M\"unchen, Zentrum Mathematik - M11 Boltzmannstra\ss e 3 85748 Garching bei M\"unchen}
	\email[A1]{laface@ma.tum.de}
	\author{Sofia Tirabassi}
	\address[A2]{Department of Mathematics, Stockholm University, Kr\"aftriket, SE-106 91 Stockholm Sweden\\
	Department of Mathematics, University of Bergen, All\'egatan 41, Bergen, Norway}
	\email[A2]{tirabassi@math.su.se}
%	\address{TU M\"unchen, Zentrum Mathematik -- M11,  Boltzmannstra{\ss}e 3, 85748 Garching bei M\"unchen (Germany)}
%	\email{laface@ma.tum.de}
	%\dedicatory{Dedicated to my mother on the occasion of her 50th birthday}

\clearpage\maketitle
\thispagestyle{empty}
\begin{abstract}
We give a notion of ordinary Enriques surfaces and their canonical lifts in any positive characteristic, and we prove Torelli-type results for this class of Enriques surfaces.
\end{abstract}

%\setcounter{tocdepth}{1}
%level -1: part, 0: chapter, 1: section, etc.
%\tableofcontents 

\section{Introduction}
The geometry of surfaces has been a very prolific area of algebraic geometry, since its birth with the Italian school at the end of the XIX century. Today, it is still one of the most active areas of research thanks to its many interdisciplinary connections. 

Among all surfaces, a very well-behaved class of surfaces is that of \textit{K3 surfaces} (for an account over $\IC$, see \cite[Ch.~VII]{barth-et_al04}). The moduli of these surfaces have been extensively studied, one of the cornerstones of this theory being the Torelli theorem for K3 surfaces. 

Let us quickly recall the key points of this important result. If $X/\IC$ is a K3 surface, a \textit{marking} of $X$ is an isometry $\phi: \rH^2(X, \Z) \longra L_\text{K3}$, where $L_\text{K3}=E_8(-1)^{\oplus 2} \oplus U^{\oplus 3}$. The everywhere regular 2-form $\omega_X \in \rH^0(X,\Omega^2_X)$ (i.e.~the \textit{period}) gives via $\phi$ a distinguished 1-dimensional subspace of $L_\text{K3}$, thereby yielding a point in $\IP(L_\text{K3})$. With respect to the cup product on $\rH^2(X,\Z)$, $\omega_X$ satisfies the following \textit{period relations}: $\omega_X.\omega_X = 0$ and $\omega_X . \overline{\omega}_X >0$. This suggests that one should define the \textit{period domain} as 
\[\kd := \lbrace [x] \in \IP(L_\text{K3} \otimes \IC) \, \vert \, x.x=0, \ x.\overline{x}>0 \rbrace.\] 
Each marked K3 surface determines a point in $\kd$, and in fact every point of $\kd$ is attained by some marked K3 surface (a result known as "surjectivity of the period map"). To get rid of the markings, one can quotient $\kd$ by the action of the group of isometries of $L_\text{K3}$, thus obtaining $\kd/\ko(L_\text{K3})$. A K3 surface $X$ now determines a \textit{period point} $\gp(X) \in \kd/\ko(L_\text{K3})$, and the Torelli theorem for K3 surfaces asserts that a K3 surface is indeed determined by its period point (see \cite{barth-et_al04} for a precise statement). 

Closely related to K3 surfaces are Enriques surfaces, which were first discovered by Enriques as an answer to the question of whether a surface with $p_g=q=0$ is necessarily rational. Every complex Enriques surface admits a 2:1 \'etale cover $\pi: X \longra Y$, where $X$ is a K3 surface equipped with a fixed point free involution $\sigma$ (called \textit{Enriques involution}) which is exactly the non-trivial Deck transformation of the covering map $\pi$. For these surfaces, a version of the Torelli theorem holds. 

Let us define the notion of period for an Enriques surface $Y$, $\pi:X \longra Y$ being its K3-cover. Choose an \textit{Enriques marking} of $X$, that is an isometry $\phi: \rH^2(X, \Z) \longra L_\text{K3}$ such that $\phi \circ \sigma^* = \rho \circ \phi$, where $\rho$ is the involution on $L_\text{K3}$ given by $\rho(x,y,z_1,z_2,z_3)=(y,x,-z_1,z_3,z_2)$ (here $x,y \in E_8(-1)$ and $z_1,z_2,z_3 \in U$). The $\rho$-invariant lattice $L^+$ is precisely the Enriques lattice $E_8(2)\oplus U(2)$. The choice of a marking is unique up to an element of the group 
\[ \Gamma := \lbrace g_{\vert L^-} : \, g \in \Aut(L), \, g \circ \rho = \rho \circ g \rbrace, \]
where $L^-$ denotes the anti-invariant lattice of $L$ with respect to $\rho$. Let us define $\kd^-:=\lbrace [x] \in \Omega : \, \rho_{\mathbb C} (x) = -x \rbrace$ to be the period domain of Enriques surfaces. The \textit{period point} of $Y$, $\mathfrak{p}(Y)$ is the image of the period point of $(X,\phi)$ in the quotient space $\kd^- / \Gamma$. The Global Torelli theorem for complex Enriques surfaces states that two Enriques surfaces are isomorphic if and only if they have the same periods.

Moving to positive characteristic, for K3 surfaces we necessarily have to make a distinction depending on whether the height is infinite or not. In the first case, we are dealing with \textit{supersingular K3 surfaces} and a Torelli theorem for such K3 surfaces is available by work of Ogus if the characteristic of the base field is $p \geq 5$. In the latter case, Nygaard proves a Torelli theorem for \textit{ordinary K3 surfaces}, that is K3 surfaces with height $h=1$ (these are general in the moduli space of K3 surfaces).

Naturally, one might wonder whether it is possible to push these results to Enriques surfaces. In this article, we want to study the case of \textit{ordinary Enriques surfaces}, that is Enriques surfaces in positive characteristic whose K3-like cover is a smooth K3 surface of height $h=1$. 

We briefly describe the content of this work. In Section \ref{sec:oes}, we discuss ordinary Enriques surfaces in any positive characteristic and we construct their canonical lifts. This shows that, for an ordinary K3 surface, the property of having an Enriques quotient is retained by its canonical lift (Theorem \ref{canonical}).

In Section \ref{sec:equiniygaard}, we give a version of Nygaard's Torelli theorem that works for (polarized) ordinary Enriques surfaces over finite fields. Although this might have been known to experts, this illustrates some of the techniques that are used later on.

Afterwards, in Section \ref{sec:periods}, we discuss a Torelli theorem for (unpolarized) ordinary Enriques surfaces (Theorem \ref{torelli}) over fields whose ring of Witt vector can be embedded into $\IC$. Our setting includes the case of finite fields as a special instance. The upshot is that an ordinary Enriques surface is determined by the period of its canonical lift.

Finally in Section \ref{sec:equiLenny}, we construct a groupoid of ordinary Enriques surfaces over finite fields and, as in Taelman's work \cite{taelman17}, we construct a fully faithful functor to a suitable category of linear algebra data. This functor turns out to be an equivalence of categories if one assumes a strong form of potential semi-stable reduction (this is condition $(\star)$ in \cite{taelman17}).

\begin{ack}
	We would like to thank Wushi Goldring, Oliver Gregory, Christian Liedtke, Gebhard Martin and Maciej Zdanowicz for their interesting insights and comments. Both authors are also grateful to Lenny Taelman for giving them valuable suggestions of a first draft of the present paper. 	ST was partially supported by the grant 261756 of the Research Councils of Norway. RL is supported by the ERC Consolidator Grant 681838 K3CRYSTAL.
\end{ack}

\begin{notation}
    Throughout the paper, we will be using the following conventions, unless otherwise specified.
    $$
\begin{array}{ll} 
  k & \mbox{ a perfect field of characteristic $p > 0$} \\
  W \equiv W(k) & \mbox{ the ring of Witt vectors over $k$ }\\
  W_n \equiv W_n(k) & \mbox{ the ring of truncated Witt vectors of length $n$ over $k$ }\\
  K & \mbox{ the fraction field of $W$}\\
  \ko_L & \mbox{ the integral closure of $W$ in $L$, $L$ being a finite extension of $K$ }\\
  \ell & \mbox{ the residue field of $\ko_L$ }\\
  \wh{\text{Br}}(\gX) & \mbox{ the formal Brauer group of a K3 surface $\gX$ }\\
  \Psi(\gX) & \mbox{ the enlarged formal Brauer group of a K3 surface $\gX$ }
 \end{array}
$$
To keep the notation consistent, we will  also occasionaly write $\ko_K$ for $W$. We will be using subscripts to indicate base changes, and we will often use rings instead of their spectra in doing so. For example, if $\kx$ is a scheme over $W$, $\kx_K$ will denote its generic fiber and $\kx_{\ko_L}$ its base change to $\ko_L$.
\end{notation}

\section{Ordinary Enriques surfaces and canonical lifts}\label{sec:oes}

Let $Y/k$ an Enriques surface over a perfect field of characteristic $p>0$. If $p$ is odd, then $Y$ admits an \'etale two-to-one cover $\pi : X\rightarrow Y$, where $X$ is K3 surfaces. If $p=2$, then the situation is more complicated (see \cite[Ch.~I, \S.3]{cossec-dolgachev89} for a complete picture) but still there is a degree two finite and flat morphism $\pi : X\rightarrow Y$, where $X$ is a \emph{K3-like} surface, i.e.~a possibly non-normal surface with trivial dualizing sheaf and with the same cohomology of a K3 surface. We are interested in Torelli-type results for Enriques surfaces in positive characteristic. To this end, we first recall that a K3 surface $X$ over a field of positive characteristic is said to be \emph{ordinary}, if it satisfies one of the following equivalent conditions:
\begin{enumerate}[label=(\roman*)]
\item the Hodge and Newton polygons of $\rH^2_{\mathrm{crys}}(X/W)$ coincide;
\item  the Frobenius endomorphism yields a bijection on $\rH^2(X,\mathcal{O}_X)$;
\item the formal Brauer group of $X$ has height $h=1$.
\end{enumerate}

By applying the notion of ordinary K3 surface to the K3-like cover of an Enriques surfaces, we can introduce the following definition:

\begin{defi}[Ordinary Enriques surface]
	An \textit{ordinary Enriques surface} is an Enriques surface whose K3-like cover is an ordinary K3 surface.
\end{defi}

\begin{rmk}
Notice that in characteristic two, an ordinary Enriques surface is precisely a $\mu_2$-Enriques surface, i.e.~an Enriques surface $Y$ such that $\Pic^\tau (Y) \cong \mu_2$. As it turns out, $\mu_2$-Enriques surfaces are usually called "ordinary" (or "singular") in the literature.
\end{rmk}

Let $Y$ be an ordinary Enriques surface over $k$, and let $X$ be its K3-cover. As $X$ is ordinary, we can lift it to $W$. Moreover, among all lifts of $X$ to $W$, the \emph{canonical lift} $\kx_\text{can}/ W$ of $X$ is particularly well-behaved (see \cite{nygaard83tate}). We will be concerned with properties of $X$ that are preserved by passing to its canonical lift $\kx_\text{can}$, and viceversa which properties of $X$ can be read off $\kx_\text{can}$.

As an example, it would be desirable for the canonical lift of an Enriques K3 surface $X$ to retain the property of having an Enriques quotient, and we will show that indeed this is the case. To this end, we begin by discussing a general observation about lifting automorphisms of ordinary K3 surfaces, which has been observed by Kaushal Srivastava \cite[Theorem 4.11]{kaushal18}:

\begin{prop}\label{kaushal}
	Every isomorphism of ordinary K3 surfaces over an perfect field of characteristic $p$ lifts to an isomorphism of their canonical lift.
\end{prop} 

In \cite{kaushal18}, the author makes restrictions on the ground field (which is assumed to be algebraically closed and of odd characteristic) and the proof makes strong use of derived categories and Taelman's result \cite[Theorem C]{taelman17}. Here, we would like to propose a different approach which uses the canonicity of the canonical lift and deformation theory.

To this end, let us briefly recall the main property of the canonical lift of an ordinary K3 surface that we will be using in the following, and we refer to \cite[\S 1]{taelman17} for a concise but exhaustive introduction. 

Let $k$ be a perfect field of characteristic $p>0$ and let $\Lambda$ being a complete local noetherian ring with residue field $k$. If $X/k$ is an ordinary K3 surface and $\gX/\Lambda$ is a formal lift of $X$, then the \'etale-local exact sequence of $\Psi(\gX)$ looks like
\[ 0 \longra \Psi(X)_\text{can}^\circ \longra \Psi(\gX) \longra \Psi(X)_\text{can}^\text{\'et} \longra 0,\]
where $\Psi(X)_\text{can}^\circ$ is the unique lift to $\Lambda$ of the unit component $\Psi(X)^\circ$ of $\Psi(X)$, and $\Psi(X)_\text{can}^\text{\'et}$ is the unique lift to $\Lambda$ of the \'etale part $\Psi(X)^\text{can}$ of $\Psi(X)$. In fact, by \cite{artin-mazur77}, $\Psi(X)^\circ \cong \wh{\text{Br}}(X)$ canonically.

It follows that we have a map
\[ \wh{\text{Def}(X)}(\Lambda) \longra \Ext^1_\Lambda(\Psi(X)_\text{can}^\text{\'et},\wh{\text{Br}}(X)_\text{can})\]
sending a formal lift $\gX$ to the extension defined above. By \cite[Theorem 1.6]{nygaard83tate}, this map is an isomorphism. Therefore, one defines the canonical lift of $X$ over $\Lambda$ to be unique lift for which the above extension splits.

\begin{proof}%Denote by $W$ and $W_n$ the rings of Witt vectors and of truncated Witt vector (of length $n$) over $k$ respectively. 
Let $X,X'$ be two ordinary K3 surfaces over $k$, and let $\phi: X \longra X'$ be an isomorphism. Consider the canonical lift $\kx_\text{can}/ W$ of $X$, and let $X_n/W_n$ to be the restriction of $\kx_{\text{can}}$ over $W_n$.  Thus we can think of $\kx_{\text{can}}$ as the algebraization of the formal scheme $\gX:=\{X_n/W_n\}$. 
	%which can be thought of a sequence of infinitesimal lifts 
	%\[X'=X'_0 \subset X'_1 \subset X'_2 \subset \cdots,\]
	%where the schemes $X'_n /  W_n$ are compatible lifts of $X'$ to the %truncated Witt vectors. 
	
	By \cite[Lemma 85.10.6]{stacks-project}  there is an isomorphism of infinitesimal deformation functors $\text{Def}_{\phi:X\rightarrow X'}\longrightarrow \text{Def}_X$, thus for every $n$ there exists a deformation $X'_n$ over $W_n$ and  a morphism $\phi_n: X_n \longra X_n'$ over $W_n$, which specializes to $\phi$. As being an isomorphism is an open condition we have that $\phi_n$ is still an isomorphism. Thus we get  an isomorphism of formal schemes $\wh{\phi}: \gX \longra \gX'$.  
	
	We want to show that $\gX'$ is algebraizable. To this aim let $A$ an ample line bundle on $X$. Since $\kx_\text{can}$ is a lift preserving the Picard group, we can consider the completion $\wh{A}$ of $A$ on $\gX:=\wh{\kx_\text{can}}$ . As $\wh{\phi}$ is an isomorphism we have that $\wh{A}=\wh{\Phi}^*\wh{M}$ for some line bundle $\wh M$ on $\gX'$.  The restriction $M$ of $\wh{M}$ to $X'$ pulls back to $A$ which is ample, thus $M$ it is ample itself. By Grothendieck's algebraization theorem, we obtain a scheme $\kx' / W$ and an isomorphism $\Phi: \kx_\text{can} \longra \kx'$ specializing to $\phi$. 
	
	Notice that $\kx'$ is a priori not the canonical lift of $X$\footnote{In fact, we could use at this point Taelman's result \cite[Theorem C]{taelman17} to immediately conclude that $\kx'=\kx'_\text{can}$, but we should look for a more natural/functorial approach.}. However, since $\wh{\text{Br}}$ and $\Psi$ are functorial (they are defined in terms of cohomology), we obtain the following diagram.
	
	\[
	\xymatrix{
	0 \ar[r] & \Psi(X)_\text{can}^\circ \ar[r] & \Psi(\gX)  \ar[r] & \Psi(X)_\text{can}^\text{\'et} \ar[r] & 0\\
		0 \ar[r] & \Psi(X')_\text{can}^\circ \ar[r] \ar[u]_{\phi^*} & \Psi(\gX') \ar[r] \ar[u]_{\phi^*} & \Psi(X')_\text{can}^\text{\'et} \ar[r] \ar[u]_{\phi^*} & 0
	}
	\]
	Therefore, the bottom extension splits if and only if the above one does, and we are done by the definition of canonical lift.
\end{proof}

Now, we apply this result to show that the property of having an Enriques quotient is transfered from an ordinary K3 surface to its canonical lift.

\begin{thm}\label{canonical}
	Let $X$ be an ordinary K3 surface over a perfect field $k$ that admits an Enriques quotient $Y$. Then, its canonical lift also admits an Enriques quotient, which is a lift of $Y$ over $W(k)$.	
\end{thm}
	
\begin{proof}
	We will prove that if $\sigma$ is an Enriques involution on $X$, then its lift to its canonical lift $\kx_\text{can}$ is also an Enriques involution. Certainly, by Proposition \ref{kaushal}, $\sigma$ lifts to an automorphism $\Sigma : \kx_\text{can} \longra \kx_\text{can}$. Therefore, we need to prove that $\Sigma$ is an involution and that it has (geometrically) no fixed points. 
	
	By an argument of Lieblich and Maulik \cite[Proof Theorem 2.1]{lieblich-maulik11}, we have that the specialization map $\text{sp}:\mathrm{Aut}(\kx_\text{can,$\ol{K}$})\rightarrow \mathrm{Aut}(X_{\ol{k}})$ is injective. In particular $\sigma$ and $\Sigma$ have the same order, hence $\Sigma$ is an involution.
	
	Finally, suppose that $\Sigma$ fixes a geometric fixed point $\overline{s}:\mathrm{Spec}(\Omega)\rightarrow \kx_{\text{can}}$, $\Omega$ being an algebraically closed field. Then, by choosing an embedding $K\hookrightarrow \Omega$, there is a finite extension $L$ of $K$ and an $L$-valued point $s:\mathrm{Spec}(L)\rightarrow \kx_{\text{can}}$ which is fixed by $\Sigma$. In turn this will yield, by the valuative criterion for properness a fixed $\ko_L$-valued point. Specializing this provides a fixed geometric point for $\sigma$, giving a contradiction.
\end{proof}

In light of the above result, the following definition makes sense.

\begin{defi}[Canonical lift of an Enriques surface]
	Let $Y$ be an ordinary Enriques surface over $k$, with K3 cover $X$ and Enriques involution $\sigma$. The \emph{canonical lift} of $Y$ is the Enriques quotient $\ky_\text{can}/ W$ of the canonical lift $\kx_\text{can}$ of $X$ by the (unique) lift $\Sigma$ of $\sigma$ to $ W$.
\end{defi}

\subsection{On the number of Enriques quotients of an ordinary K3 surface}

In this paragraph, we take a small detour and use Nygaard's theory of canonical lift for ordinary K3 surfaces to count the number of their Enriques quotients. In order to do that we have to fix an embedding $\iota:W\hookrightarrow \mathbb{C}$, thus we restrict ourselves to perfect fields whose cardinality is at most that of  $\mathbb{C}$. Given a K3 surface $X$ over $k$, we will denote by $X_{\mathrm{can}}^\iota$ the complex K3 obtained from $\mathcal{X}_{\mathrm{can}}$ by base changing through $\iota$. In this notation we have the following statement.
\begin{prop}\label{prop: Enriques}
Let $X$ be an ordinary K3 surface over $k$, $\text{char}(k) \neq 2$. The number of Enriques quotients of $X_{{k}}$ is equal the number of Enriques quotients of its canonical lift $X^\iota_{\mathrm{can}}$
\end{prop}
\begin{proof}
We are going to show that there is a one-to-one correspondence between involutions of $X$ which do not fix any geometric point and the primitive embeddings of the lattice ${E}_{10}:=H(2)\oplus E_8(2)$ into $NS(\mathcal{X}^\iota_{\mathrm{can}})$ such that
\begin{itemize}
    \item[(A)] The orthogonal of $E_{10}$ has no  vector with square equal -2
    \item[(B)] $E_{10}$ contains an ample divisor.
\end{itemize}
The result will follow from the work of Ohashi \cite{Oha2007}, who proves a similar result over $\IC$. Let $\sigma$ an involution of $X_{\overline{k}}$ without fixed geometric points. By the results of the first section we can extend it to a free fixed point involution of $\mathcal{X_\mathrm{can}}/ W$ which in turns specializes to a free fixed point involution of $X^\iota_\mathrm{can}$. By \cite[Prop. 2.2]{Oha2007}, there is a uniquely determined primitive embedding of ${E}_{10}$ into $\NS(X^\iota_\mathrm{can})$  satisfying the condition (A) and (B) above. \par
Conversely if we have such a primitive embedding, then we get (in a unique way) an Enriques involution $\Sigma$ on $X^\iota_\mathrm{can}$. We will follow the proof of \cite[Theorem 2.5]{Jang2015} to show that this specializes to an Enriques involution on $X$. We remark that $\Sigma$ is defined over  a finite (intermediate) extension $L$ of $K$. Then $\Sigma$ extends to an involution $\overline{\Sigma}$ of $\mathcal{X}_{\mathcal{O}_L}$. Denote by $\sigma$ the specialization of $\overline{\Sigma}$ to the closed fiber. By Matsusaka--Mumford this is an involution of $X$. We see that $\sigma^*$ acts as multiplication by $-1$ on $\rH^0(X,\omega_X)$ as $\overline{\Sigma}$ does the same. We now claim that $\sigma$ is an Enriques involution. Indeed, assuming the existence of fixed points, \'etale locally around a fixed point\footnote{Here is where we need odd characteristic.} the representation of $\sigma= \Sigma \vert_{X_{\bar{k}}}$ is given by the matrix $\begin{pmatrix} 1 & 0 \\ 0 & -1\end{pmatrix}$. Therefore, there is a curve of fixed points and thus the quotient surface $Y:=X / \sigma$ is smooth. Since the characteristic is odd, the surface $Y$ can have Kodaira dimension at most zero. Also, $Y$ must be minimal (there are no (-1)-curves) since $\Pic(Y) = E_{10}$ is an even lattice. By a case-by-case analysis, $Y$ is not isomorphic to any of the minimal surfaces with vanishing Kodaira dimension, thus it must be ruled. However, the Picard lattice of minimal ruled surfaces has rank at most two, contradicting $\rho(Y)=10$. This shows that $\sigma$ is an Enriques involution, and we are done.
\end{proof}
\begin{rmk}\label{rmk: Enriques}
The proof of the above statement shows how, over finite fields, one can count the number of Enriques quotients of an ordinary K3 surface from the linear algebra data provided by Taelman in \cite{taelman17}. More precisely, to any ordinary K3 surface over a finite field $k$, we can associate a triplet $(M, F, K)$, where $M:= \rH^2(\mathcal{X}_{\mathrm{can}}^\iota, \mathbb{Z})$, $F:M\rightarrow M$ is a lift of Frobenius and $K$ is the ample cone in $\NS(\mathcal{X}_{\mathrm{can}}^\iota)$. The proof of \ref{prop: Enriques} shows that The number of Enriques quotients of $X$ is equal to the number of primitive embeddings of the lattice Enriques lattice $E_{10}=E_8(2)\oplus H(2)$ into $\NS(\mathcal{X}^\iota_{\mathrm{can}})$ such that the conditions (A) and (B) above hold.
\end{rmk}

\section{Torelli theorem for polarized Enriques surfaces over finite fields}\label{sec:equiniygaard}
In this paragraph, we study the Torelli problem for ordinary Enriques surfaces in positive characteristic by using the result of the previous section and Nygaard's Torelli theorem for polarized K3 surfaces over finite fields (see \cite{nygaard83}).

 We will work over $\mathbb{F}_q$ ($q = p^r$, $r>0$). After fixing an embedding $\iota:W\rightarrow \mathbb{C}$, we can associate to any ordinary K3 surface $X/\IF_q$ a complex K3 surface, $X_\mathrm{can}^\iota$ by base-changing the canonical lift $\mathcal{X}_\mathrm{can}\rightarrow W$ along $\iota$. In addition, there is an endomorphism $F$ of $\rH^2(X^\iota_\mathrm{can},\mathbb{Z}[\frac{1}{p}])$ such that for every prime $\ell \neq p$ the canonical comparison isomorphism $\rH^2(X_\text{can}^\iota,\Z_\ell) \longra \rHet{2}(X_{\ol{\IF}_q}, \Z_\ell)$ matches $F$ with (the pull-back of) the geometric Frobenius $F_X$ on \'etale cohomology. In this setting Nygaard's Torelli theorem can be stated as follows:

\begin{thm}[Nygaard]\label{Nygaard}
	Let $(X_1,\km_1)$ and $(X_2,\km_2)$ be two ordinary K3 surfaces over $\IF_q$. Assume that:
		\begin{enumerate}
			\item there exists an isometry $\phi : \rH^2(X_{2,\mathrm{can}}^\iota,\Z) \longra \rH^2(X_{1,\mathrm{can}}^\iota,\Z)$ that is compatible with the polarizations;
			\item there exists a positive integer $m > 0$ such that $\phi \circ F_{X_1}^m = F_{X_2}^m \circ \phi$.
		\end{enumerate}		
	Then, $(X_1,\km_1) \cong (X_2,\km_2)$ over $\overline{\IF}_q$.
\end{thm}
The results of the previous section give us a theory of canonical lift for ordinary Enriques surface. Therefore, to any ordinary Enriques surface $Y$ over $k$ we can associate a complex Enriques surface $Y_\mathrm{can}^\iota$ by base-changing its canonical lift $\mathcal{Y}_\mathrm{can}\rightarrow W$. In this notation we get the following Torelli-type theorem for polarized ordinary Enriques surfaces, whose proof paves the way towards our results in the rest of this paper.

\begin{thm}\label{polarized}
	Let $(Y_1,\kl_1)$ and $(Y_2,\kl_2)$ be two ordinary Enriques surfaces over $\IF_q$. Denote by $(X_1,\km_1)$ and $(X_2,\km_2)$  respectively their  K3 covers  equipped with the pull-back polarization. Assume that:
		\begin{enumerate}
		\item there exists an isometry $\phi : \rH^2(X_{2,\mathrm{can}}^\iota,\Z) \longra \rH^2(X_{1,\mathrm{can}}^\iota,\Z)$ that is compatible with the polarizations;
			\item there exists a positive integer $m > 0$ such that $\phi \circ F_{X_1}^m = F_{X_2}^m \circ \phi$.
			\item the isometry $\phi$ is compatible with the Enriques involutions on the canonical lifts, that is, $\Sigma_{1,\IC}^* \circ \phi = \phi \circ \Sigma_{2,\IC}^*$, $\Sigma_1$ and $\Sigma_2$ being lifts of the Enriques involution to the canonical lifts.
		\end{enumerate}		
	Then, $(Y_1,\kl_1) \cong (Y_2,\kl_2)$ over $\overline{\IF}_q$.
\end{thm}

\begin{proof}
By Theorem \ref{Nygaard} we have that the polarized K3 surfaces $(X_1, M_1)$ and $(X_2, M_2)$ are isomorphic over $\overline{\IF}_q$. However, by looking at Nygaard proof of Theorem \ref{Nygaard}, one sees that the isomorphism above comes from an isomorphism of integral models. 

More precisely,  let $\mathcal{M}_1$ and $\mathcal{M}_2$ be the lift of $M_1$ and $M_2$ to $\kx_{1,\mathrm{can}}$ and $\kx_{2,\mathrm{can}}$, respectively, and denote by $M_i^\iota$ the restriction of $\mathcal{M}_i$ to $X_{i\:\mathrm{can}}^\iota$.  By the Torelli theorem for complex K3 surfaces, there exists $\psi : (X^\iota_{1,\mathrm{can}},M^\iota_1) \longra (X^\iota_{2,\mathrm{can}}, M^\iota_2)$ inducing $\phi$ (that is, $\psi^* = \phi$). The isomorphism $\psi$ is defined over a finite extension $L$ of $K = \text{Frac}(W)$. If we denote by $\ko_L$ the normalization of $W$ in $L$, we obtain two polarized $\ko_L$-models $\mathcal{X}_{1,\ko_L}:=(\kx_{1,\mathrm{can}},\km_1) \times_W \ko_L$ and $\mathcal{X}_{2,\ko_L}:=(\kx_{2,\mathrm{can}},\km_2)\times_W \ko_L$ whose generic fibers are isomorphic as polarized varieties. The Matsusaka-Mumford criterion \cite{matsusaka-mumford64} then implies that the isomorphism $\Psi_L$ on the generic fiber extends to one on the $\ko_L$-models, say $\Psi_{\ko_L}$, and therefore to one on the central fibers, say $\Psi_\ell$.

	By the condition on the isometry $\phi$ being compatible with the Enriques involutions, we can conclude by the argument above that $\Psi_L \circ \Sigma_L = \Sigma'_L \circ \Psi_L$. We claim that the same holds over $\ko_L$, that is we have the following commutative diagram.
\[
\xymatrix{
\kx_{1,\ko_L} \ar[r]^{\Psi_{\ko_L}} \ar[d]_{\Sigma_{1,\ko_L}} & \kx_{2,\ko_L} \ar[d]^{\Sigma_{2,\ko_L}} \\
 \kx_{1,\ko_L} \ar[r]_{\Psi_{\ko_L}} & \kx_{2,\ko_L}
}
\]
By \cite[Theorem 2.1]{lieblich-maulik11} and the construction of a lift in Proposition \ref{kaushal}, the restriction morphism
\[ \Aut(\kx_{i,\ko_L}) \longra \Aut(\kx_{i,L}), \qquad \Phi \longmapsto \Phi_\eta \]
is an isomorphism (here $\eta$ is the generic point of $\Spec \ko_L$). This implies that also the morphism 
\[ \Isom(\kx_{1,\ko_L},\kx_{2, \ko_L}) \longra \Isom(\kx_{1,L},\kx_{2,L})\]
is an isomorphism, showing that $\Psi_{L}$ lifts to a unique $\Psi_{\ko_L}$. Therefore, as $\Psi_L^{-1} \circ \Sigma_{2,L} \circ \Psi_L \circ \Sigma_{1,L}= \id_{\kx_{1,L}}$, it follows that $\Psi_{\ko_L}^{-1} \circ \Sigma_{2,\ko_L} \circ \Psi_{\ko_L} \circ \Sigma_{1,\ko_L}= \id_{\kx_{1,\ko_L}}$ as claimed.

By passing to the central fibers, this implies that $\Psi_\ell \circ \sigma_1 = \sigma_2 \circ \Psi_\ell$. This means that the isomorphism is compatible with the Enriques involutions also on the special fibers, and thus that the Enriques surfaces $Y$ and $Y'$ are isomorphic over $\ell$, which is a finite extension of $k$.
\end{proof}

\section{Torelli theorem for unpolarized Enriques surfaces}\label{sec:periods}
In this section we aim at proving a Global Torelli Theorem for (unpolarized) ordinary Enriques surfaces over a perfect field $k$. In order to run our argument we will need to fix an embedding $\iota:W\rightarrow \mathbb{C}$, thus we can only consider fields $k$ whose cardinality is at most that of $\mathbb{C}$. Our strategy consists in considering the canonical lift of a given ordinary Enriques surface, and base change via $\iota$ to $\IC$, where we know the result holds true. For sake of readability, we will not keep track of $\iota$ in our notation, although it will always be understood that one such embedding of $W$ into $\IC$ is fixed.

Enriques surfaces have a particularly well-behaved deformation theory for line bundles: as $\rH^2(\ko_Y) = \rH^1(\ko_Y)=0$, there is no obstruction to lifting a line bundle and the lift is unique. Therefore, for a lift $\ky/W$ of $Y$, $\Pic(Y) \cong \Pic(\ky)$. Since we ultimately want to use the Matsusaka-Mumford criterion, we will need to extend ample line bundles on the generic fiber $\ky_K$ to the integral model $\ky/W$. 

However, although an ample line bundle on $\ky_K$ does indeed extend to a line bundle on $\ky$, it does not necessarily extend to an ample line bundle. Indeed, this can occur whenever the central fiber $Y$ of $\ky$ has additional $(-2)$-curves with respect to the generic fiber. Let us now see when we can overcome this difficulty.

\begin{lemma}\label{technical}
 	Let $Y/k$ be an Enriques surface, let $\ky/W$ be a lift of $Y$ to the Witt vectors and let $\kl_K$ be an ample line bundle on the generic fiber $\ky_K$.
 	\begin{enumerate}
 		\item If $Y$ is an unnodal Enriques, then $\kl_K$ extends to an ample line bundle $\kl$ on $\ky$, thus to an ample line bundle $\kl_k$ on $Y$.
 		\item Let $\kx$ be the K3 cover of $\ky$, and assume $X:= \kx \times_\ky Y$ is the K3-cover\footnote{Therefore, if $p=2$, we are again assuming that $Y$ is an ordinary Enriques surface.} of $Y$. If $\kx$ is a N\'eron-Severi lift of $X$, then $\kl_K$ extends to an ample line bundle $\kl$ on $\ky$, thus to an ample line bundle $\kl_k$ on $Y$.
	\end{enumerate} 	 
\end{lemma}

\begin{proof}
The first statement follows from the discussion above. As for the second claim, consider the following commutative diagram. 
\[
\xymatrix{
X \ar[r] \ar[d]_{\pi} & \kx_\text{can} \ar[d]^{\Pi} & \kx_\text{can,$K$} \ar[l] \ar[d]^{\Pi_K} \\
Y \ar[r] & \ky & \ky_\text{can,$K$} \ar[l]
}
\]
Let $\sigma$ (respectively $\Sigma$, $\Sigma_K$) be the Enriques involution on $X$ (respectively $\kx_\text{can}$, $\kx_\text{can,$K$}$). Then, $\km_K := \Pi_K^*\kl_K$ is an ample line bundle on $\kx_\text{can,$K$}$ which is also $\Sigma_K$-invariant. Under the isomorphism $\Pic(\kx_\text{can}) \longra \Pic(\kx_\text{can,$K$})$ (or rather its inverse), $\km_K$ extends to a line bundle $\km$ on $\kx_\text{can}$, which is $\Sigma$-invariant and ample by \cite[Proposition 2.4]{lieblich-maulik11}. Therefore, it specializes to an ample $\sigma$-invariant line bundle $\km_k$ on X, which descends to a line bundle $\kl_k$ on $Y$. By construction, this is the specialization of $\kl_K$, and we are done.
\end{proof}

We will prove a Torelli type theorem for Enriques surfaces in positive characteristic that satisfy the conditions of Lemma \ref{technical}.

\begin{thm}\label{torelli}
	Let $Y_1$ and $Y_2$ be two Enriques surfaces satisfying the conditions of Lemma \ref{technical}, and let $\ky_1$ and $\ky_2$ be their respective lifts to $W$. If $\gp(\ky_{1,\IC}) = \gp(\ky_{2,\IC})$, then $Y_1 \cong Y_2$ after possibly performing a finite extension of the ground field.
\end{thm}

\begin{proof}
	The proof is a combination of the Torelli theorem for Enriques surfaces over $\mathbb C$ and the Matsusaka-Mumford criterion.

If the two Enriques surfaces have the same period points, it follows that there exist special markings $\phi_1: \rH^2(\kx_{1,\mathbb C},\Z) \longra L$ and $\phi_2: \rH^2(\kx_{2,\mathbb C},\Z) \longra L$ on the corresponding K3 covers such that the composite
	\[ \Phi:= \phi_1^{-1} \circ \phi_2 : \rH^2(\kx_{2,\mathbb C},\Z)  \longra \rH^2(\kx_{1,\mathbb C},\Z)\]
	is a Hodge isometry. 
	
	If $\Sigma_1$ and $\Sigma_2$ denote the Enriques involutions on $\kx_{1,\IC}$ and $\kx_{2,\IC}$, respectively, fix a $\Sigma_2^*$-invariant ample polarization $\km_{2,\IC}$ on $\kx_{2,\mathbb C}$ and consider the class $\km_{1,\IC}:=\Phi(\km_{2,\IC})$ on $\kx_{1,\mathbb C}$. Up to the action of an element of the Weyl group, we can assume that $\km_{1,\IC}$ is ample and $\Sigma_1^*$-invariant on $\kx_{1,\mathbb C}$, and thus that $\Phi$ is an effective Hodge isometry. By the Torelli theorem for complex K3 surfaces, there exists $g: \kx_{1,\mathbb C} \longra \kx_{2,\mathbb C}$ such that $g^*= \Phi$. In particular, $g$ is an isomorphism of polarized K3 surfaces.
	
	Now, $(g \circ \Sigma_1 \circ g^{-1} \circ \Sigma_2)^*$ is the identity on $\rH^2(\kx_{2,\mathbb C},\Z)$, hence $g \circ \Sigma_1 \circ g^{-1} \circ \Sigma_2=\text{id}_{\kx_{2,\mathbb C}}$. This implies that $g$ descends to an isomorphism $f_\IC: \ky_{1,\mathbb C} \longra \ky_{2,\mathbb C}$ of the Enriques quotients. Notice that $f_\IC$ is an isomorphism of the polarized Enriques surfaces $(\ky_{1,\mathbb C}, {\kl}_{1,\IC})$ and $(\ky_{2,\mathbb C}, {\kl}_{2,\IC})$, $\kl_{1,\IC}$ and $\kl_{2,\IC}$ being the quotient polarization: indeed, the polarizations on the K3 surfaces descend to the Enriques quotients, as they are invariant under the corresponding Enriques involutions.
	
	Naturally, everything is defined over a finite extension $L/K$. Let $\ko_L$ be the integral closure of $W$ in $L$, and let $\ell$ be its residue field, which is a finite extension of $k$. The $\Sigma_{1,L}^*$-invariant polarization $\km_{1,L}$ extends (uniquely) to a line bundle $\km_{1,\ko_L}$ on $\kx_{1,\ko_L}$, which is $\Sigma_1^*$-invariant: this follows from the isomorphism $\Pic(\kx_{1,\ko_L}) \cong \Pic(\kx_{1,L})$ and the commutative diagram below.
\[
\xymatrix{
\kx_{1,L} \ar[r]^{\Sigma_{1,L}} \ar[d] & \kx_{1,L} \ar[d] \\
 \kx_{1,\ko_L} \ar[r]_{\Sigma_{1,\ko_L}} & \kx_{1,\ko_L}
}
\]
By Lemma \ref{technical} $\km_ {1,\ko_L}$ is ample, thus it descends to an ample line bundle $\kl_{1,\ko_L}$ on $\ky_{1,\ko_L}$, and in turn this yields a polarization on $Y_{1,\ell}$. The same considerations can be made for $Y_2$.

We are now in the hypothesis of the Matsusaka-Mumford theorem: we conclude that the isomorphism of polarized pairs $f_L: (\ky_{1,L},\kl_{1,L}) \longra (\ky_{2,L}, \kl_{2,L})$ extends to an isomorphism of $\ko_L$-models $f_{\ko_L}: \ky_{1,\ko_L} \longra \ky_{2,\ko_L}$, yielding an isomorphism $f_\ell : Y_{1,\ell} \longra Y_{2,\ell}$ upon passing to the central fibers.
\end{proof}

As a corollary, we see that the isomorphism class of an ordinary Enriques surface is determined by that of its canonical lift, as it happens for K3 surfaces.

\begin{cor}
	Let $Y_1$ and $Y_2$ be ordinary Enriques surfaces over $k$. Then, $Y_\text{$1$,can} \cong Y_\text{$2$,can}$ if and only if $(\ky_{1,\text{can}})_{\bar{K}} \cong (\ky_{2,\text{can}})_{\bar{K}}$.
\end{cor}

\begin{proof}
	The "if" follows from Theorem \ref{torelli}. As to the "only if" implication, we can assume that the isomorphism in the statement is defined over $k$. Let $\phi: Y_1 \longra Y_2$ be this isomorphism, and let $\psi: X_1 \longra X_2$ be its lift to the K3-covers. Then, $\psi$ commutes with the Enriques involutions, and so does its lift $\Psi: \kx_\text{$1$,can} \longra \kx_\text{$2$,can}$ to the canonical lifts. Therefore, $\Psi$ descends to an isomorphism $\Phi: \ky_\text{$1$,can} \longra \ky_\text{$2$,can}$.
\end{proof}

\section{Ordinary Enriques surfaces over finite fields \`a la Taelman}\label{sec:equiLenny}

We now  want to apply the results of the previous paragraph to give a description of  Enriques surfaces over finite fields by using linear algebra data in a similar way to what Taelman does in \cite{taelman17} for K3 surfaces. Let $Y$ be an Enriques surface over a finite field $\IF_q$, and choose, like before, an embedding $i:W\rightarrow \mathbb{C}$. Let again $\mathcal{Y}_\text{can}^\iota$  be the complex Enriques surfaces obtained from $\mathcal{Y}_\text{can}$ via extension of scalars through $\iota$.

We now introduce the following two groupoids: 
$$
\begin{array}{ll} 
    \underline{\ke n}^\text{ord} & \mbox{ the groupoid of ordinary Enriques surfaces, whose objects are ordinary } \\
    & \mbox{ Enriques surfaces over $\IF_q$ and whose arrows are isomorphisms between them;} \\
    & \\
    \underline{\kk 3}^\text{ord} & \mbox{ the groupoid of ordinary K3 surfaces, whose objects are ordinary K3 surfaces }\\
    & \mbox{ over $\IF_q$ and whose arrows are isomorphisms between them.}
\end{array}
$$

In order to associate to an ordinary Enriques surface over a finite field some linear algebra data, we need to first go through the groupoid of ordinary K3 surfaces, which we know maps fully faithfully onto a suitable groupoid of linear algebra structures by \cite{taelman17}. Naively, one is let to define a functor $\underline{\ke n}^\text{ord} \longra \underline{\kk 3}^\text{ord}$ by sending an Enriques surface to its K3-cover. Unfortunately, this functor is not going to be fully faithful, and consequently neither will its composition with a suitable functor $\underline{\kk 3}^\text{ord} \longra \underline{\kh\ks}$, where here $ \underline{\kh\ks}$ is a suitable category of linear algebra data ("HS" stands for "Hodge Structures"). Indeed, not every isomorphism of K3 surfaces induces an isomorphism of the corresponding Enriques quotients. 

Moreover, any isomorphism $\phi: Y \longra Y'$ between two Enriques surfaces lifts to an isomorphism $\widehat{\phi}: X \longra X'$ of the corresponding K3-covers which is compatible with both Enriques involutions. However, $\phi$ is also induced by the isomorphisms $\phi \circ \sigma$ and $\sigma' \circ \phi$, $\sigma$ and $\sigma'$ being the Enriques involutions on $X$ and $X'$, so that $F$ is still not fully faithful.

Nevertheless, observe that given an isomorphism $\phi: Y \longra Y'$ of Enriques surfaces, its lift $\widehat{\phi}: X \longra X'$ to the K3-covers is unique up to twisting by the Enriques involutions. This suggest to consider the following groupoid:
$$
\begin{array}{ll} 
    \underline{\ke n \kk 3}^\text{ord} & \mbox{ the groupoid of pairs $(X,\sigma)$, where $X$ is an ordinary K3 surface over $\IF_q$, }\\
    & \mbox{ $\sigma$ is an Enriques involution, and whose arrows are non-ordered pairs}\\
    & \mbox{ $(f, f \circ \sigma)$, where $f$ is an isomorphism compatible with both involutions.}
\end{array}
$$

Notice that the compatibility condition says that $(f, f \circ \sigma)=(f, \sigma' \circ f)$, if $f:X \longra X'$ is an isomorphism as above. Also, given two composable pairs of arrows $(f,f \circ \sigma):(X,\sigma) \longra (X',\sigma')$ and $(g,g \circ \sigma'): X' \longra X''$, the composition $(g,g\circ \sigma') \circ (f,f \circ \sigma)$ is obtained as follows: consider all four possible compositions 
\[ g \circ f, \quad g \circ f \circ \sigma, \quad  g \circ \sigma' \circ f \quad \text{and} \quad  g \circ \sigma' \circ f \circ \sigma,\]
and notice that the compatibility conditions yields a unique pair of arrows $(g \circ f, g \circ f \circ \sigma): (X,\sigma) \longra (X'', \sigma'')$.

Now, we observe that the lift $\Sigma_K$ of the Enriques involution $\sigma$ on $X/\IF_q$ to the canonical lift $\kx_\text{can,K}$ commutes with the lift $F_{\kx_\text{can,K}}$ of the Frobenius endomorphism $F_X$ on $X$. This seems known to experts, however we could not find a source so we decide to include a short argument for the reader convenience. To this aim, note first that, as $X$ is defined over a finite field, the absolute Frobenius is a morphism of $\IF_q$-varieties and thus it commutes with $\sigma$, i.e.~$\sigma \circ F_X = F_X \circ \sigma$. By \cite{lieblich-maulik11}, the specialization map $\text{sp}: \Aut(\kx_{\text{can,K}}) \longra \Aut(X_{\ol{k}})$ is injective, whence $ \Sigma_K \circ F_{\kx_{\text{can,K}}}= F_{\kx_{\text{can,K}}} \circ \Sigma_K$.

In light of this, we can define a groupoid $\underline{\kh\ks}$ of linear algebra data, which will be our target category. Its objects are the quadruplets $(M,F,K,i)$ where:

\begin{itemize}
    \item the triple $(M,F,K)$ is as in \cite{taelman17};
    \item $i$ is an involution of the lattice $M$, such that 
            \begin{itemize}
                \item $i \circ F = F \circ i$;
                \item $\text{Fix}(i) \cong E_{10}$;
                \item $\text{Fix}(i) \subset \NS(M,F)$ and this inclusion is a primitive embedding;
                \item there is no element $\delta \in \NS(M,F)$, $\delta^2=-2$, orthogonal to $\text{Fix}(i)$;
                \item $(K \cap \NS(M,F)) \cap \text{Fix}(i) \neq \emptyset$, i.e.~$\text{Fix}(i)$ contains an integral ample class.
            \end{itemize}
\end{itemize}
The morphisms in $\underline{\kh\ks}$ are non-ordered pairs $(f,i \circ f)$ of morphisms as in \cite{taelman17} such that $f$ is compatible with both involutions. We have now two functors
\[ F: \underline{\ke n}^\text{ord} \longra \underline{\ke n \kk 3}^\text{ord} \qquad \text{and} \qquad G: \underline{\ke n \kk 3}^\text{ord} \longra \underline{\kh\ks},\]
which are defined as follows: $F$ sends an Enriques surface $Y$ to the pair $(X,\sigma)$ consisting of its K3-cover together with the corresponding Enriques involution, and $G$ sends $(X,\sigma)$ to the quadruple
\[ \big( \rH^2(X_\text{can}^\iota,\Z), \text{Frob}_{X_\text{can}^\iota}^*, \text{Amp}(X_{\ol{\IF}_q}), \sigma^* \big) .\]

With these definition, one can prove the following result.

\begin{thm}
    The functor $G$ is fully faithful.
\end{thm}

\begin{proof}
    The functor is clearly well-defined on objects. On morphisms, the condition on the lift of Frobenius to the canonical lift commuting with the Enriques involution implies that it is also well-defined on morphisms just by pull-back. Let $(X,\sigma)$ and $(X',\sigma')$ be two objects in  $\underline{\ke n \kk 3}^\text{ord}$. Then, the map 
    \[ \Hom_{ \underline{\ke n \kk 3}^\text{ord}} \big( (X,\sigma),(X',\sigma') \big) \longra \Hom_{\underline{\kh\ks}} \big( F(X',\sigma'),F(X,\sigma) \big)\]
    is injective, as there are no non-trivial cohomologically trivial automorphisms on K3 surfaces. If $(\Phi, \Psi) \in \Hom_{\underline{\kh\ks}} \big( F(X',\sigma'),F(X,\sigma) \big)$, then they are both effective Hodge isometries (see \cite{nygaard83}), and therefore the Torelli theorem implies that there exist isomorphisms $f,g:X \longra X'$ such that $f^*=\Phi$ and $g^*=\Psi = \sigma^* \circ \Phi$. In particular, $g = f \circ \sigma$ and the above map is also surjective.
\end{proof}

Let us now discuss essential surjectivity of $G$. In order to do this, we need to assume a strong form of potential semi-stable reduction: this is condition $(\star)$ in \cite{taelman17}. 

\begin{prop}
    Assume condition $(\star)$ holds for all K3 surfaces over $K$ that admit an Enriques quotient. Then the functor $G$ above is essentially surjective.
\end{prop}

\begin{proof}
    Let $(M,F,K,i)$ be an object in $\underline{\kh\ks}$. The triplet $(M,F,K)$ corresponds to a K3 surface $X / \IF_q$ by \cite{taelman17}. By definition, $i: (M,F,K) \longra (M,F,K)$ is a morphism of triplets (recall that $i$ commutes with $F$). Therefore, by full faithfulness in \cite{taelman17}, there exists an involution $\sigma: X \longra X$ such that $\sigma^*=i$. Now, $\text{Fix}(\sigma^*)$ is the Enriques lattice $E_{10}$, and thus $\sigma$ is an Enriques involution by \cite{Jang2015}. We conclude that the pair $(X,\sigma)$ maps to the quadruplet $(M,F,K,i)$, and we are done.
\end{proof}
\begin{rmk}
Alternatively we could have observed that our assumptions on $(M, F, i)$ and the work of Ohashi \cite{Oha2007} imply that the canonical lift of $X/\mathbb{F}_q$ is a Enriques- K3 surface with Enriques involution uniquely determined by the embedding of the fixed locus of $i$ in $NS(M,F)$. Then we can use the proof of Proposition \ref{prop: Enriques} to conclude.
\end{rmk}
%%%%%%%%%%%%%%%%%%%%%%%%%%%%%%%%%%%%%%%%%%%% 
\bibliographystyle{plain}
\bibliography{bib}{}

%\noindent
%\hrulefill

%\vspace{5mm}
%\SMALL{
%\noindent \textsc{Dr. Roberto Laface}
%\\Technische Universit\"at M\"unchen
%\\Zentrum Mathematik - M11
%\\85748 Garching bei M\"unchen (Germany)
%\\ \textit{E-mail address: }{\tt laface@ma.tum.de}
%\\ 
%\vspace{0.5cm} 
%\noindent \urladdr{https://old.i2m.univ-amu.fr/~roulleau.x/}
%}

%\newpage

\end{document}